\documentclass[11pt]{article}
\usepackage{indentfirst}
\usepackage[utf8]{inputenc}
\usepackage{amssymb,latexsym,amsmath,amsbsy,amsthm,amsxtra,amsgen,graphicx,makeidx,dsfont,longtable,bbm,float,mathtools}
\usepackage{booktabs}
\usepackage{hyperref}
\usepackage{titlesec}
\usepackage{enumitem}
\usepackage{tocloft}
\usepackage{caption}
\usepackage{hyperref}
\usepackage{tikz}

\newtheorem{remark}{Remark}[section]

\newtheorem{lemma}{Lemma}[section]

\newtheorem{theorem}{Theorem}[section]

\begin{document}
	\title{\textbf{Limiting Spectral Distribution of High-dimensional Multivariate Kendall-$\tau$}}
	\author{Ruoyu Wu}
    \date{}
	\maketitle
	
	\vspace{-.6in}
	\begin{abstract}
The multivariate Kendall-$\tau$ statistic, denoted by $K_n$, plays a significant role in robust statistical analysis. This paper establishes the limiting properties of the empirical spectral distribution (ESD) of $K_n$. We demonstrate that the ESD of $\frac{1}{2}pK_n$ converges almost surely to the Marčenko--Pastur law with variance parameter $\frac{1}{2}$, analogous to the classical result for sample covariance matrices. Using Stieltjes transform techniques, we extend these results to the independent component model, deriving a fixed-point equation that characterizes the limiting spectral distribution of $\frac{1}{2}tr\Sigma K_n$. The theoretical findings are validated through comprehensive simulation studies.
	\end{abstract}
	
	\section{Introduction}
	\subsection{Background}
A well-unknown and significant result in Random Matrices Theory is about the spectral behavior of sample covariance matrices, which implies their eigenvalue distributions converge to the Marčenko–Pastur law under the hypothesis $p/n \to y \in (0,1)$.

In practical applications, heavy-tailed data and outliers can compromise the stability of moment-based estimators such as the sample covariance. Rank-based dependence measures, including Spearman’s $\rho$ and Kendall’s $\tau$, provide robust alternatives that do not rely on moment assumptions. Extending random matrix theory to these correlation-type matrices—particularly the multivariate Kendall–$\tau$ matrix—offers valuable insights into robust high-dimensional statistics.

For correlation matrices, Jiang \cite{Jiang2004} established asymptotic properties of eigenvalues for large correlation matrices. Rank-based correlation matrices have attracted increasing attention due to their robustness. Bai and Zhou \cite{BaiZhou2008} showed that the empirical spectral distribution of the Spearman correlation matrix follows the Marčenko–Pastur law. Bandeira, Lodhia, and Rigollet \cite{Bandeira2017} further proved that the sample Kendall–$\tau$ correlation matrix exhibits an eigenvalue distribution that is a linear transformation of the Marčenko–Pastur law. More recently, Wu and Wang \cite{WuWang2022} extended these findings to dependent samples and established convergence results under weak dependence.

The multivariate Kendall–$\tau$ statistic was first introduced by Choi and Marden \cite{ChoiMarden1998} on a generalization of Kendall’s $\tau$ for testing independence between multivariate samples. Han and Liu \cite{HanLiu2016} later utilized it to develop robust estimators for elliptical distributions. 

This thesis investigates the asymptotic spectral behavior of the sample multivariate Kendall–$\tau$ matrix $K_n$. It establishes that the empirical spectral distribution of 
$\frac{1}{2}pK_n$ converges to the Marčenko–Pastur law with variance parameter $\frac{1}{2}$. The analysis further extends to an independent-component model $X_i = \mu + \Sigma^{1/2} Z_i$
and derives the corresponding Stieltjes transform equation for the limiting distribution. Using moment and Stieltjes-transform methods, the study demonstrates a deep connection between the spectral behavior of multivariate Kendall–$\tau$ matrices and classical covariance-type matrices, advancing the theoretical foundation for robust high-dimensional inference.

\subsection{Notation and Preliminaries}
Let $X_i = (X_{i1}, X_{i2}, \ldots, X_{ip})^T$ $(i=1,2,\ldots,n)$ denote $n$ independent observations from a $p$-dimensional random vector $X = (X_1, X_2, \ldots, X_p)^T$.  
Throughout this chapter, we assume that the dimension $p$ grows with $n$ such that the ratio $p/n \to y \in (0,1)$ as $n \to \infty$.

We define the sample multivariate Kendall-$\tau$ matrix as
\begin{equation}
K_n = \frac{2}{n(n-1)} \sum_{1 \le i < j \le n}
\frac{(X_i - X_j)(X_i - X_j)^T}{\|X_i - X_j\|^2}.
\end{equation}

Let $\lambda_1, \lambda_2, \ldots, \lambda_p$ be the eigenvalues of $K_n$.  
The \emph{empirical spectral distribution} (ESD) of $K_n$ is defined by
\begin{equation}
F^{K_n}(x) = \frac{1}{p} \sum_{i=1}^{p} \mathbbm{1}(\lambda_i \le x),
\end{equation}
where $\mathbbm{1}(\cdot)$ is the indicator function.  
We are interested in the limiting distribution of $F^{K_n}$ as $n, p \to \infty$.

\medskip

\noindent
\textbf{Notation.}
For a matrix $A = (a_{ij})_{m \times n}$:
\begin{itemize}[itemsep=3pt, leftmargin=1.2cm]
  \item $\|A\| = \sup_{\|x\|_2 = 1} \|A x\|_2$ denotes the spectral norm;
  \item $\|A\|_F = (tr(A^T A))^{1/2}$ denotes the Frobenius norm;
  \item $trA$ denotes the trace of $A$;
  \item $A^T$ denotes the transpose of $A$;
  \item For two matrices $A$ and $B$, the notation $A \sim B$ means their empirical spectral distributions are asymptotically equivalent.
\end{itemize}

For two distribution functions $F$ and $G$, define the \emph{L\'evy distance} as
\begin{equation}
L(F, G) = \inf\{\varepsilon > 0: F(x-\varepsilon) - \varepsilon \le G(x)
\le F(x+\varepsilon) + \varepsilon, \ \forall x \in \mathbb{R}\}.
\end{equation}

For a Hermitian matrix $A$, the \emph{Stieltjes transform} of its ESD $F^A$ is defined as
\begin{equation}
m_A(z) = \int \frac{1}{\lambda - z}\, dF^A(\lambda)
= \frac{1}{p} \mathrm{tr}[(A - zI_p)^{-1}], \qquad z \in \mathbb{C}^+,
\end{equation}
where $\mathbb{C}^+ = \{z \in \mathbb{C} : \mathrm{Im}(z) > 0\}$.

The limiting spectral distribution (LSD) of $\{A_n\}$ is said to exist if $F^{A_n}(x)$ converges weakly to a nonrandom distribution $F(x)$ as $n,p \to \infty$.  
Its Stieltjes transform $m_F(z)$ satisfies $m_{A_n}(z) \to m_F(z)$ for all $z \in \mathbb{C}^+$.

\section{Main Results}
\begin{theorem}[The limiting spectral distribution in the i.i.d. case]\label{thm2.1}
Let $\{X_{ij}\}$ $(i=1,\ldots,p; j=1,\ldots,n)$ be independent and identically distributed random variables with mean $0$, variance $1$, and finite fourth moment.  
Define $X_i = (X_{1i}, X_{2i}, \ldots, X_{pi})^T$.  
If $p/n \to y \in (0,1)$, then the empirical spectral distribution of $\tfrac{1}{2}pK_n$ converges almost surely to the Marčenko–Pastur distribution with variance parameter $\tfrac{1}{2}$, whose density is
\begin{equation}
f_y(x) = 
\begin{cases}
\dfrac{1}{2\pi xy}\sqrt{(b-x)(x-a)}, & a \le x \le b, \\[1.2em]
0, & \text{otherwise,}
\end{cases}
\end{equation}
where $a = \tfrac{1}{2}(1 - \sqrt{y})^2$ and $b = \tfrac{1}{2}(1 + \sqrt{y})^2$.
\end{theorem}

\begin{remark}
Theorem~\ref{thm2.1} shows that after normalization by $\tfrac{1}{2}p$, the eigenvalue distribution of $K_n$ follows the same limiting law as that of the sample covariance matrix.  
This implies that the rank-based multivariate Kendall-$\tau$ statistic retains a spectral structure consistent with that of covariance-type estimators.
\end{remark}

\begin{theorem}[General model with independent components]\label{thm2.2}
Let $Z_i = (Z_{i1}, Z_{i2}, \ldots, Z_{ip})^T$ be random vectors whose components are i.i.d.\ with mean $0$, variance $\tfrac{1}{2}$, and finite fourth moment.  
Suppose $X_i = \mu + \Sigma^{1/2} Z_i$, where $\Sigma$ is a $p \times p$ Hermitian positive-definite matrix with uniformly bounded spectral norm.  
Let $H_p$ denote the empirical spectral distribution of $\tfrac{1}{2}\Sigma$, and assume $H_p \to H$ as $p \to \infty$.  
If $p/n \to y \in (0,1)$, then the empirical spectral distribution of $\tfrac{1}{2}tr\Sigma K_n$ converges almost surely to a nonrandom limit distribution $F$ whose Stieltjes transform $m_F(z)$ satisfies
\begin{equation}\label{eq:stieltjes_main}
m_F(z) = \int \frac{1}{\tau(1 - y - y z m_F(z)) - z}\, dH(\tau), \qquad z \in \mathbb{C}^+.
\end{equation}
\end{theorem}

The proofs of Theorems~\ref{thm2.1} and~\ref{thm2.2} will be provided in the subsequent sections.  
They rely on a sequence of auxiliary lemmas that establish almost sure convergence of certain matrix traces and control the difference between $K_n$ and sample covariance-type matrices.

\section{Proofs}
\subsection{Auxiliary Lemmas}

In order to prove Theorems~\ref{thm2.1} and~\ref{thm2.2}, we require several auxiliary results.  
These lemmas establish almost sure convergence for pairwise summations and control of partial sums.

\begin{lemma}[Pairwise Convergence Lemma]\label{lem2.1}
Let $\{X_i\}_{i=1}^n$ be a sequence of i.i.d.\ random variables with $\mathbb{E}[X_i] = \mu$ and $\mathbb{E}[X_i^4] < \infty$.  
Then
\begin{equation}
\frac{1}{\binom{n}{2}} \sum_{1 \le i < j \le n} X_i X_j \xrightarrow{a.s.} \mu^2.
\end{equation}
\end{lemma}

\begin{proof}
Observe that
\[
\sum_{1 \le i < j \le n} X_i X_j
= \frac{1}{2} \left( \left(\sum_{i=1}^{n} X_i\right)^2 - \sum_{i=1}^{n} X_i^2 \right).
\]
By the strong law of large numbers (SLLN),
\[
\frac{1}{n} \sum_{i=1}^{n} X_i \xrightarrow{a.s.} \mu, \qquad
\frac{1}{n} \sum_{i=1}^{n} X_i^2 \xrightarrow{a.s.} \mathbb{E}[X^2].
\]
Dividing both sides by $\binom{n}{2}$ and noting that $\frac{n^2}{\binom{n}{2}} \to 2$, we have
\[
\frac{1}{\binom{n}{2}} \sum_{1 \le i < j \le n} X_i X_j
= \frac{1}{2} \left(\frac{n}{n-1}\right) \left[ \left( \frac{1}{n}\sum_{i=1}^{n} X_i \right)^2
- \frac{1}{n} \left( \frac{1}{n}\sum_{i=1}^{n} X_i^2 \right) \right] \xrightarrow{a.s.} \mu^2.
\]
\end{proof}

\begin{lemma}[Maximal Partial Sum Lemma]\label{lem2.2}
Let $\{X_{ij}\}$ be i.i.d.\ random variables.  
For constants $\alpha > \tfrac{1}{2}$, $\beta \ge 0$, and integer $M > 0$, we have
\begin{equation}
\max_{1 \le j \le M n^{\beta}}
\left| n^{-\alpha} \sum_{i=1}^{n} (X_{ij} - c) \right| \xrightarrow{a.s.} 0
\end{equation}
if and only if $\mathbb{E}|X_{11}|^{(1+\beta)/\alpha} < \infty$,  
where
\[
c = 
\begin{cases}
\text{any real number}, & \text{if } \alpha > 1, \\
\mathbb{E}[X_{11}], & \text{if } \alpha \leq 1.
\end{cases}
\]
\end{lemma}

\begin{proof}

\textbf{Sufficiency:}  
Without loss of generality, let \( c = 0 \). For \( \epsilon > 0 \) and \( N \geq 1 \),

\[
P \left\{ \max_{j \leq M n^\beta} \left| n^{-\alpha} \sum_{i=1}^n (X_{ij} - c) \right| \geq \epsilon, \text{ i.o.} \right\} 
\]
\[
\leq \sum_{k \geq N} P \left\{ \max_{2^{k-1} < n \leq 2^k} \max_{j \leq M 2^{k\beta}} \left| \sum_{i=1}^n X_{ij} \right| \geq 2^{k\alpha} \epsilon' \right\}
\]
\[
\leq \sum_{k \geq N} M 2^{k\beta} P \left\{ \max_{2^{k-1} < n \leq 2^k} \left| \sum_{i=1}^n X_{i1} \right| \geq 2^{k\alpha} \epsilon' \right\}
\]

Here, \( \epsilon' = 2^{-\alpha} \epsilon \). By the Borel–Cantelli lemma, it suffices to prove:

\[
\sum_{k=1}^{\infty} 2^{k\beta} P \left\{ \max_{n \leq 2^k} \left| \sum_{i=1}^n X_{i1} \right| \geq 2^{k\alpha} \epsilon' \right\} < \infty
\]

Now, define \( Y_k = X_{i1} I\left( |X_{i1}| < 2^{k\alpha} \right) \) and \( Z_{ik} = Y_k - \mathbb{E}[Y_k] \). Then \( Z_{ik} \) is bounded by \( 2^{k\alpha+1} \) and \( \mathbb{E}[Z_{ik}] = 0 \), thus forming a martingale difference sequence. Let \( g > \frac{\beta + 2\alpha}{\alpha - 1/2} \) be a sufficiently large positive even integer. By the submartingale inequality:

\[
P \left\{ \max_{n \leq 2^k} \left| \sum_{i=1}^n Z_{ik} \right| \geq 2^{k\alpha} \epsilon' \right\} 
\leq C 2^{-kg\alpha} \mathbb{E} \left[ \left| \sum_{i=1}^{2^k} Z_{ik} \right|^g \right]
\]

So \[
P \left\{ \max_{n \leq 2^k} \left| \sum_{i=1}^n Z_{ik} \right| \geq 2^{k\alpha} \epsilon' \right\}
\leq C 2^{-kg\alpha} \left\{ 2^k \mathbb{E}[|Z_{ik}|^g] + 2^{kg/2} \left( \mathbb{E}[Z_{ik}^2] \right)^{g/2} \right\}
\]

Note that

\begin{align*}
\sum_{k=1}^{\infty} 2^{k\beta - kg\alpha + k} \mathbb{E}[Z_{ik}^g] &\leq C \sum_{k=1}^{\infty} 2^{k\beta - kg\alpha + k} \mathbb{E}[X_{i1}^g I(|X_{11}| < 2^{k\alpha})] \\
&\leq C \sum_{k=1}^{\infty} \mathbb{E}\left[|X_{11}|^{\frac{1+\beta}{\alpha}}\right] I(2^{\alpha l-1} < |X_{11}| < 2^{\alpha l}) + C_0 \\
&< \infty
\end{align*}

Similarly, we can prove:
\[
\sum_{k=1}^{\infty} 2^{k\beta - kg\alpha + k} (\mathbb{E}[Z_{ik}^2])^{g/2} < \infty
\]

Therefore, we have proved:
\[
\sum_{k=N}^{\infty} 2^{k\beta} P \left\{ \max_{n \leq 2^k} \left| \sum_{i=1}^n Z_{ik} \right| \geq 2^{k\alpha} \epsilon \right\} < \infty
\]

Now since \( Z_{ik} \) is just a standardized version of \( Y_{ik} \), we only need to estimate the bound of \( \mathbb{E}[Y_{ik}] \):
\[
\max_{n \leq 2^k} \left| \sum_{i=1}^n \mathbb{E}[Y_{ik}] \right| \leq 2^k |\mathbb{E}[Y_{1k}]| \leq 0.5 \epsilon 2^{k\alpha}
\]

Clearly, we also have:
\[
\sum_{k=N}^{\infty} 2^{k\beta} P \left\{ \max_{n \leq 2^k} \left| \sum_{i=1}^n Y_k \right| \geq 2^{k\alpha} \epsilon \right\} < \infty
\]

Finally, noting that \( \mathbb{E}[|X_{11}|^{\frac{1+\beta}{\alpha}}] < \infty \), we obtain:
\[
\sum_{k=1}^{\infty} 2^{k\beta} P \left\{ \bigcup_{i=1}^{2^k} {|X_{i1}|\geq 2^{k\alpha}}\right\}   \leq \sum_{k=1}^{\infty} 2^{k(\beta+1)} P \left\{ |X_{11}| \geq 2^{k\alpha}  \right\} < \infty
\]

This completes the proof of sufficiency.

\textbf{Necessity:}

If \( \beta = 0 \), it is the result of the Marcinkiewicz SLLN. If \( \beta > 0 \), we have the condition:
\[
\max_{j \leq M(n-1)^\beta} n^{-\alpha} \left| \sum_{i=1}^n X_{ij} \right| \overset{\text{a.s.}}{\to} 0
\]

Then we have:
\[
\max_{j \leq Mn^\beta} n^{-\alpha} |X_{j}| \overset{\text{a.s.}}{\to} 0
\]

By the Borel-Cantelli lemma, we have:
\[
\sum_{n=1}^{\infty} P\left( \max_{j \leq M n^\beta} |X_{j}| \geq n^\alpha \right) < \infty
\]

This expression is equivalent to:
\[
\prod_{n=1}^{\infty} P\left( \max_{j \leq M n^\beta} |X_{j}| < n^\alpha \right) = \prod_{n=1}^{\infty} \left(1 - P(|X_{11}| \geq n^\alpha)\right)^{\lfloor M n^\beta \rfloor} > 0
\]

Therefore, we have:
\[
\sum_{n=1}^{\infty} M n^\beta P(|X_{11}| \geq n^\alpha) < \infty
\]

Hence, \( \mathbb{E}|X_{11}|^{\frac{\beta+1}{\alpha}} < \infty \). This completes the proof of necessity.
\end{proof}

\subsection{Proofs of the Main Theorems}
\subsubsection{Proof of Theorem~\ref{thm2.1}}
To prove Theorem~\ref{thm2.1}, we consider demonstrating that a sample covariance matrix and a statistic related to the multivariate Kendall's tau matrix share similar limiting spectral distributions. 

According to the conditions in Theorem~\ref{thm2.1}, for independent random samples \(X_i\), where each random sample \(X_i\) consists of components that are independent and identically distributed \(p\)-dimensional vectors \(X_{i1}, X_{i2}, \cdots, X_{ip}\).

Thus, \(\left \{ X_{jk} \right \}_{j,k = 1, 2, \ldots}  \) actually forms a double sequence of independent and identically distributed real random variables with finite fourth moments. 

One may assume their mean is 0 and variance is \(\frac{1}{2}\). Later we will see that this variance does not affect the limiting spectral distribution of the statistic. According to the definition of the sample covariance matrix:
\begin{align*}
S_n 
&= \frac{1}{n-1} \sum_{i=1}^n (X_i - \bar{X})(X_i - \bar{X})^T\\
&= \frac{1}{n(n-1)} \sum_{1 \leq i < j \leq n} (X_i - X_j)(X_i - X_j)^T
\tag{2.1}
\end{align*}

Note that the following equivalent form is closely related to the multivariate Kendall's tau matrix, so in subsequent calculations we will directly start from this form. Define the straightened matrix \(X_{ij}\) with \(p \times \binom{n}{2}\) elements:
\[
X_{(i,j)} = (X_1 - X_2, X_1 - X_3, \cdots, X_2 - X_3, \cdots, X_{n-1} - X_n)
\]

That is: each column of this matrix sequentially maps the pair \((i, j)\) to \(X_i - X_j\), with all these pairs sorted in lexicographical order.

We examine the decomposition of the covariance matrix:

\[
S_n = \frac{1}{2\binom{n}{2}} \sum_{1 \leq i < j \leq n} (X_i - X_j)(X_i - X_j)^T = \frac{1}{2\binom{n}{2}} X_{(i,j)} X_{(i,j)}^T
\]

Similarly, we define \(Y_{(i,j)}\) as the normalized sample differences after the same lexicographical ordering:

\[
Y_{(i,j)} = \left( \frac{X_1 - X_2}{\|X_1 - X_2\|}, \frac{X_1 - X_3}{\|X_1 - X_3\|}, \cdots, \frac{X_2 - X_3}{\|X_2 - X_3\|}, \cdots, \frac{X_{n-1} - X_n}{\|X_{n-1} - X_n\|} \right)
\]

Thus we have:

\[
\frac{1}{2} pK_n = \frac{p/2}{\binom{n}{2}} \sum_{1 \leq i < j \leq n} \frac{(X_i - X_j)(X_i - X_j)^T}{\|X_i - X_j\|^2} = \frac{p/2}{\binom{n}{2}} Y_{(i,j)} Y_{(i,j)}^T\tag{2.2}
\]

We wish to prove that \(pK_n / 2\) and \(S_n\) have the same limiting spectral distribution. It suffices to prove that the Lévy distance between their empirical spectral distributions \(F^{S_n}\) and \(F^{pK_n / 2}\) tends to 0.

We cite the following lemma without proof \cite{Bai1999}.

\begin{lemma}[Estimation of \emph{L\'evy distance}]\label{lem2.3}
Suppose \(A\) and \(B\) are two \(p \times n\) complex matrices. Then:
\[
L^4(F^{AA^T}, F^{BB^T}) \leq \frac{2}{p^2} \text{tr}((A - B)(A - B)^T) \cdot \text{tr}(AA^T + BB^T)
\]
\end{lemma}

Therefore, to prove \(L(F^{S_n}, F^{pK_n / 2}) = L(F^{\frac{1}{2}X_{(i,j)}X_{(i,j)}^T}, F^{\frac{p/2}{\binom{n}{2}}Y_{(i,j)}Y_{(i,j)}^T}) \to 0\), it suffices to prove:
\[
\frac{1}{p} \text{tr}\left( \left( \sqrt{\frac{1}{2\binom{n}{2}}} X_{(i,j)} - \sqrt{\frac{p/2}{\binom{n}{2}}} Y_{(i,j)} \right) \left( \sqrt{\frac{1}{2\binom{n}{2}}} X_{(i,j)} - \sqrt{\frac{p/2}{\binom{n}{2}}} Y_{(i,j)} \right)^T \right) \triangleq \Delta_n \to 0\tag{2.3}
\]

Calculation shows:

\begin{align*}
\Delta_n &= \frac{1}{p} \text{tr}\left( \sqrt{\frac{1/2}{n\binom{n}{2}}}X_{(i,j)} - \sqrt{\frac{p/2}{n\binom{n}{2}}}Y_{(i,j)} \right) \left( \sqrt{\frac{1/2}{n\binom{n}{2}}}X_{(i,j)} - \sqrt{\frac{p/2}{n\binom{n}{2}}}Y_{(i,j)} \right)^{T} \\
&= \frac{1/2}{p\binom{n}{2}} \text{tr}\left( X_{(i,j)} - \sqrt{p}Y_{(i,j)} \right) \left( X_{(i,j)} - \sqrt{p}Y_{(i,j)} \right)^{T} \\
&= \frac{1/2}{p\binom{n}{2}} \left\| X_{(i,j)} - \sqrt{p}Y_{(i,j)} \right\|_{F}^{2} \\
&= \frac{1/2}{p\binom{n}{2}} \left[ \sum_{1\leq i<j\leq n} \sum_{k=1}^{p} (X_{ki}-X_{kj})^{2} - 2\sqrt{p} \sum_{1\leq i<j\leq n} \sum_{k=1}^{p} \frac{(X_{ki}-X_{kj})^{2}}{\|X_i-X_j\|} + p \sum_{1\leq i<j\leq n} \sum_{k=1}^{p} \frac{(X_{ki}-X_{kj})^{2}}{\|X_i-X_j\|^{2}} \right] \\
&= \frac{1/2}{p\binom{n}{2}} \left[ \sum_{1\leq i<j\leq n} \sum_{k=1}^{p} (X_{ki}-X_{kj})^{2} - 2\sqrt{p} \sum_{1\leq i<j\leq n} \|X_i-X_j\| + p\binom{n}{2} \right] \\
&= \frac{1/2}{p\binom{n}{2}} \left[ \sum_{1\leq i<j\leq n} \sum_{k=1}^{p} (n-1)X_{ki}^{2} - 2 \sum_{1\leq i<j\leq n} \sum_{k=1}^{p} X_{ki}X_{kj} - 2\sqrt{p} \sum_{1\leq i<j\leq n} \|X_i-X_j\| + p\binom{n}{2} \right] \\
&= \frac{1}{2} \left[ \frac{2}{np} \sum_{i} \sum_{k=1}^{p} X_{ki}^{2} - \frac{2}{p\binom{n}{2}} \sum_{1\leq i<j\leq n} \sum_{k=1}^{p} X_{ki}X_{kj} - \frac{2}{\binom{n}{2}\sqrt{p}} \sum_{1\leq i<j\leq n} \|X_i-X_j\| + 1 \right]
\end{align*}

According to the strong law of large numbers,

\[
\frac{1}{np} \sum_{i} \sum_{k=1}^{p} X_{ki}^{2} \overset{\text{a.s.}}{\rightarrow} \mathbb{E}[X_{ki}^{2}] = \frac{1}{2}
\]

According to the result of Lemma~\ref{lem2.1}, we have:

\[
\frac{1}{p\binom{n}{2}} \sum_{1\leq i<j\leq n} \sum_{k=1}^{p} X_{ki}X_{kj} \overset{\text{a.s.}}{\rightarrow} 0
\]

Finally,

\[
\min_{(i,j)} \sqrt{\frac{\sum_{k}(X_{ki}-X_{kj})^{2}}{p}} \leq \frac{1}{\binom{n}{2}\sqrt{p}} \sum_{1\leq i<j\leq n} \|X_i-X_j\| \leq \max_{(i,j)} \sqrt{\frac{\sum_{k}(X_{ki}-X_{kj})^{2}}{p}}
\]

and according to Lemma~\ref{lem2.2},
\[
\min_{(i,j)} \frac{\sum_{k}(X_{ki}-X_{kj})^{2}}{p} \overset{\text{a.s.}}{\rightarrow} \mathbb{E}[(X_{ki}-X_{kj})^{2}] = 1
\]

\[
\max_{(i,j)} \frac{\sum_{k}(X_{ki}-X_{kj})^{2}}{p} \overset{\text{a.s.}}{\rightarrow} \mathbb{E}[(X_{ki}-X_{kj})^{2}] = 1
\]

Thus we can easily see that

\[
\frac{1}{\binom{n}{2}\sqrt{p}} \sum_{1\leq i<j\leq n} \|X_i-X_j\| \overset{\text{a.s.}}{\rightarrow} 1
\]

In summary, we obtain the result of Equation (2.3):

\[
\Delta_n = \frac{1}{2} \left[ \frac{2}{np} \sum_{i=1}^n \sum_{k=1}^p X_{ki}^2 - \frac{2}{p\binom{n}{2}} \sum_{1 \leq i < j \leq n} \sum_{k=1}^p X_{ki}X_{kj} - \frac{2}{\binom{n}{2}\sqrt{p}} \sum_{1 \leq i < j \leq n} \|X_i - X_j\| + 1 \right] \overset{\text{a.s.}}{\to} 0
\]

This implies that \(\frac{1}{2} p K_n\) and \(S_n\) have the same limiting spectral distribution. According to the result of Marčenko and Pastur, \(S_n\) follows the Marčenko-Pastur law with variance parameter 0.5.

We also note that if we let \(X_i' = \sigma X_i\), then:

\begin{align*}
K_n' &= \frac{2}{n(n-1)} \sum_{1 \leq i < j \leq n} \frac{(X_i' - X_j')(X_i' - X_j')^T}{\|X_i' - X_j'\|_2^2} \\
&= \frac{2}{n(n-1)} \sum_{1 \leq i < j \leq n} \frac{\sigma^2(X_i - X_j)(X_i - X_j)^T}{\|\sigma(X_i - X_j)\|_2^2} \\
&= K_n
\end{align*}

Therefore, we can assert that regardless of how the variance of each component of \(X_i\) changes, the limiting spectral distribution of \(\frac{1}{2} p K_n\) remains unchanged and equals the limiting spectral distribution of the sample covariance matrix when the component variance of \(X_i\) is \(\frac{1}{2}\). Thus, when \(p/n \to y \in (0,1)\), let \(a = 0.5(1 - \sqrt{y})^2\), \(b = 0.5(1 + \sqrt{y})^2\), the limiting spectral distribution of \(\frac{1}{2} p K_n\) is a distribution function with probability density function on \([a,b]\):

\[
p_y(x) = \frac{1}{\pi x y} \sqrt{(b-x)(x-a)}
\]
with zero probability density outside this interval.
\bigskip

\subsubsection{Proof of Theorem~\ref{thm2.2}}

Consider the general model
\[
X_i = \mu + \Sigma^{1/2} Z_i,
\]
where $Z_i = (Z_{i1}, Z_{i2}, \ldots, Z_{ip})^T$ are independent with
$\mathbb{E}Z_{ij} = 0$, $\mathrm{Var}(Z_{ij}) = \tfrac{1}{2}$,
and $\mathbb{E}|Z_{ij}|^4 < \infty$.
We will derive the limiting spectral distribution of $\tfrac{1}{2}tr\Sigma K_n$.

Since \(\Sigma/2\) corresponds to the covariance matrix of \( X_i \), which is Hermitian, we can easily prove through spectral decomposition properties that \(\Sigma^{\frac{1}{2}}\) is also a Hermitian matrix. For convenience, we denote it as \( A \), assuming \( A = (a_{ij})_{p \times p} \). The conditions indicate that the spectral norm \( \|A\|^2 \) of \( A \) is also bounded.

According to Equations (2.1) and (2.2), both \( S_n \) and \( K_n \) are independent of the expectation of \( X_i \). Therefore, we may assume \( X_i = \Sigma^{\frac{1}{2}} Z_i \).

Next, we will prove that \( \Sigma^{\frac{1}{2}} K_n \) and \( S_n \) have the same limiting spectral distribution.

We state the following lemma \cite{HornJohnson1985}:

\begin{lemma}\label{lem2.4}
Suppose \( A \) and \( B \) are two \( p \times n \) Hermitian matrices. Then:
\[
L(F^A, F^B) \leq \|A - B\|
\]
\end{lemma}

Indeed, by Lemma~\ref{lem2.4},
\begin{align*}
L(F^{\frac{1}{2} tr\Sigma K_n}, F^{S_n}) 
&\leq \left\| \frac{1}{2}tr\Sigma K_n - S_n \right\|\\
&= \left\| \frac{1}{2\binom{n}{2}} X_{(i,j)} \text{diag}_{(i,j)} \left( \frac{tr\Sigma}{\|X_i - X_j\|^2} \right) X_{(i,j)}^T - \frac{1}{2\binom{n}{2}} X_{(i,j)} X_{(i,j)}^T \right\|\\
&\leq \frac{1}{2\binom{n}{2}} \left\| X_{(i,j)} X_{(i,j)}^T \right\| \cdot \left\| \text{diag}_{(i,j)} \left( \frac{tr\Sigma}{\|X_i - X_j\|^2} \right) - \text{diag}(I_{\binom{n}{2}\times\binom{n}{2}} )\right\|\\
& \leq \|S_n\| \cdot \left\| \text{diag}_{(i,j)} \left( \frac{tr\Sigma}{\|X_i - X_j\|^2} \right) -  \text{diag}(I_{\binom{n}{2}\times\binom{n}{2}} ) \right\|   
\end{align*}
Here, each entry of \( \text{diag}_{(i,j)} (f) \) is sequentially \( f(X_i, X_j) \), ordered lexicographically for pairs \( (i, j) \) with \( 1 \leq i < j \leq n \). Meanwhile,
\begin{align*}
\|S_n\| 
&= \left\| \frac{1}{2\binom{n}{2}}\sum_{1 \leq i < j \leq n} (X_i - X_j)(X_i - X_j)^T \right\|\\
&= \left\| \frac{1}{2\binom{n}{2}} \sum_{1 \leq i < j \leq n} A(Z_i - Z_j)(Z_i - Z_j)^T A^T \right\|\\
&\leq \left\| \frac{1}{2\binom{n}{2}} \sum_{1 \leq i < j \leq n} (Z_i - Z_j)(Z_i - Z_j)^T \right\| \cdot \|AA^T\|
\end{align*}
Since \( \|AA^T\| \) is bounded, and after truncating \( Z_i \) with respect to constants, we have \( \|S_n\| = O(1) \). Therefore, we only need to prove:

\[
\left\| \text{diag}_{(i,j)} \left( \frac{tr\Sigma}{\|X_i - X_j\|^2} \right) -  \text{diag}(I_{\binom{n}{2}\times\binom{n}{2}} ) \right\| \overset{\text{a.s.}}{\to} 0 \tag{2.4}
\]

Since \( |a^{-1} - b^{-1}| = \frac{|a-b|}{|ab|} \), under the truncated condition, Equation (2.4) is equivalent to:

\[
\left\| \text{diag}_{(i,j)} \left( \frac{\|X_i - X_j\|^2}{tr\Sigma} \right) -  \text{diag}(I_{\binom{n}{2}\times\binom{n}{2}} ) \right\| \overset{\text{a.s.}}{\to} 0 \tag{2.5}
\]

We first consider truncating and standardizing all samples, which does not affect the asymptotic properties of the statistics.

\medskip
\noindent
\textbf{Step 1. Truncation}

Let:
\[
\tilde{Z}_{ki} = Z_{ki} I(|Z_{ki}| \leq \delta_p' (np)^{\frac{1}{4}}) = Z_{ki} I(|Z_{ki}| \leq 0.5\delta_p p^{\frac{1}{2}})
\]

Here we require:
\[
\lim_{p \to \infty} \delta_p = 0, \quad \lim_{p \to \infty} \delta_p^{-4} \mathbb{E}[|Z_{11}|^4 I(|Z_{11}| > 0.5\delta_p p^{\frac{1}{2}})] = 0, \quad \delta_p p^{\frac{1}{2}} \to \infty
\]

Note that \( p/n \to y \in (0,1) \). Using Chen's  techniques (Chen and Pan,2012, page 1409) we can easily see:
\[
P\left( \limsup_{p \to \infty} |Z \neq \tilde{Z}| \right) = 0
\]

For \( \tilde{X}_i = A \tilde{Z}_i \), this directly shows:
\[
\left\| \text{diag}_{(i,j)} \left( \frac{\|X_i - X_j\|^2}{\text{tr} \Sigma} \right) - \text{diag}_{(i,j)} \left( \frac{\|\tilde{X}_i - \tilde{X}_j\|^2}{\text{tr} \Sigma} \right) \right\| \overset{\text{a.s.}}{\to} 0
\]

This truncation method effectively achieves:
\[
|\tilde{Z}_{ki} - Z_{ki}| \leq \delta_p p^{\frac{1}{2}}
\]

\medskip
\noindent
\textbf{Step 2. Normalization}

Let:
\[
\bar{Z}_{ki} = \frac{\tilde{Z}_{ki} - \mathbb{E}[\tilde{Z}_{ki}]}{\sqrt{2\text{Var}[\tilde{Z}_{ki}]}}
\]

From the construction in Step 1, we can easily see: \( \mathbb{E}[\tilde{Z}_{11}] \to 0 \) and \( \text{Var}[\tilde{Z}_{11}] \to \frac{1}{2} \). Furthermore, since \( Z_{11} \) has finite fourth moments,
\begin{align*}
|\mathbb{E}[\tilde{Z}_{11}]|
&= |\mathbb{E}[Z_{11}] - \mathbb{E}[\tilde{Z}_{11}]| \\
&= |\mathbb{E}[Z_{11}I(|Z_{11}| > 0.5\delta_p p^{\frac{1}{2}})]| \\
&\leq \mathbb{E}\left[ |Z_{11}| \left| \frac{Z_{11}}{0.5\delta_p p^{\frac{1}{2}}} \right|^3 I(|Z_{11}| > 0.5\delta_p p^{\frac{1}{2}}) \right]\\
&= o(p^{-\frac{3}{2}})
\end{align*}

And 
\[ |\text{Var}[\tilde{Z}_{11}] - \tfrac{1}{2}| = |\mathbb{E}[Z_{11}^{2}I(|Z_{11}| > 0.5\delta_p p^{\frac{1}{2}})] + o(p^{-3})| = o(p^{-1}) \]

Through simple calculations we can obtain :
\[ \left\| \text{diag}_{(i,j)} \left( \frac{\|X_i - X_j\|^2}{\text{tr}\Sigma} \right) - \text{diag}_{(i,j)} \left( \frac{\|\tilde{X}_i - \tilde{X}_j\|^2}{\text{tr}\Sigma} \right) \right\| \overset{\text{a.s.}}{\to } 0\]

According to the triangle inequality of norms, in the subsequent proof, we only need to consider Equation (2.5) with the truncated and standardized $\tilde{Z}_i$ replacing $Z_i$. For notational convenience, we still use $Z_i$ to denote the truncated sample. 

After truncation and standardization, we have:
\[ \mathbb{E}[Z_{11}] = 0, \quad \mathbb{E}[Z_{11}^2] = \tfrac{1}{2}, \quad \mathbb{E}[Z_{11}^4] < \infty, \quad |Z_{11} - \tilde{Z}_{11}| \leq \delta_p p^{\frac{1}{2}} \]

Note that $\text{diag}_{(i,j)} \left( \frac{\|X_i - X_j\|^2}{\text{tr}\Sigma} \right) -  \text{diag}(I_{\binom{n}{2}\times\binom{n}{2}} )$ is a diagonal matrix, and its spectral norm is the maximum eigenvalue; since the eigenvalues of a diagonal matrix are its diagonal elements, we only need to consider the almost sure convergence of the maximum diagonal element. Equation (2.5) is equivalent to:
\[ \max_{(i,j)} \left| \frac{\|X_i - X_j\|^2}{\text{tr}\Sigma} - 1 \right| \overset{\text{a.s.}}{\to} 0 \tag{2.6} \]

Now we have:
\begin{align*}
&\quad \max_{(i,j)} \left| \frac{\|X_i - X_j\|^2}{\text{tr}\Sigma} - 1 \right| \\
&= \max_{(i,j)} \left| \frac{\|A(Z_i - Z_j)\|^2}{\text{tr}\Sigma} - 1 \right| \\
& = \max_{(i,j)} \left| \frac{\sum_{k=1}^p \left[ \sum_{l=1}^p a_{kl}(Z_{li} - Z_{lj}) \right]^2}{\sum_{k=1}^p \sum_{l=1}^p a_{kl}^2} - 1 \right|\\
&= \max_{(i,j)} \left| \frac{\sum_{k=1}^p \sum_{l=1}^p a_{kl}^2 (Z_{li} - Z_{lj})^2 + \sum_{k=1}^p \sum_{l=1}^p \sum_{m \neq l} a_{kl} a_{km} (Z_{li} - Z_{lj})(Z_{mi} - Z_{mj})}{\sum_{k=1}^p \sum_{l=1}^p a_{kl}^2} - 1 \right| \\
&\leq \Delta_1 + 2\Delta_2 
\end{align*}

where,
\[ \Delta_1 = \max_{(i,j)} \left| \frac{\sum_{k=1}^p \sum_{l=1}^p a_{kl}^2 (Z_{li} - Z_{lj})^2}{\sum_{k=1}^p \sum_{l=1}^p a_{kl}^2} - 1 \right| \]
\[ \Delta_2 = \max_{(i,j)} \left| \frac{\sum_{k=1}^p \sum_{m\neq l}^{p} a_{kl} a_{km} (Z_{li} - Z_{lj})(Z_{mi} - Z_{mj})}{\sum_{k=1}^p \sum_{l=1}^p a_{kl}^2} \right| \]

We only need to prove $\Delta_1 \overset{\text{a.s.}}{\to} 0$ and $\Delta_2 \overset{\text{a.s.}}{\to} 0$.

\medskip
\noindent
\textbf{Step 3. Convergence of $\Delta_1$ and $\Delta_2$}

Since \( c^{-1} \leq \sum_{k=1}^{p} a_{ik}^2 \leq \|A\|^2 \leq c \), without loss of generality, we can assume \(\sum_{i=1}^{p} a_{ik}^2 = 1\); otherwise, we simply divide all elements of the corresponding column of \(A\) by this quantity. Then:

\[
\Delta_1 = \max_{(i,j)} \left| \frac{\sum_{k=1}^{p} (Z_{ki} - Z_{kj})^2}{p} - 1 \right|
\]

By Lemma~\ref{lem2.2}, we easily obtain \(\Delta_1 \overset{\text{a.s.}}{\to} 0\).

To prove \(\Delta_2 \overset{\text{a.s.}}{\to} 0\), we first cite the following Lemma \cite{ErdosYau2017} :

\begin{lemma}\label{lem2.5}
Suppose \(X_1, X_2, \cdots, X_N\) are independent centered random variables, and for fixed constants \(u_q\) we have:
\[
(\mathbb{E}[|X_i|^q])^{\frac{1}{q}} \leq u_q, \quad 1 \leq i \leq N, \quad q = 2, 3, \cdots
\]
Then for any double sequence of complex numbers \(\{a_{ij}\}_{i,j=1}^N\),
\[
\left( \mathbb{E}\left[ \left| \sum_{i \neq j = 1}^N a_{ij} X_i X_j \right|^q \right] \right)^{\frac{1}{q}} \leq C q u_q \left( \sum_{i \neq j = 1}^N |a_{ij}|^2 \right)
\]
where \(C\) is a constant independent of \(q\).
\end{lemma}

Define \(Y_{(i,j),l} = \sum_{k_1 \neq k_2} a_{lk_1} a_{lk_2} (Z_{k_1 i} - Z_{k_1 j}) (Z_{k_2 i} - Z_{k_2 j})\)

Then \(\Delta_2 = \max_{(i,j)} \frac{1}{p}\left| \sum_{l=1}^{p}Y_{(i,j),l} \right|\). 
It is easy to see that \(\mathbb{E}[Y_{(i,j)},l] = 0\).

According to the definition of truncation:
\[
\left( \mathbb{E}[|Z_{ki} - Z_{kj}|^q] \right)^{\frac{1}{q}} \leq \left( \mathbb{E}[|Z_{ki} - Z_{kj}|^2] (\delta_p p^{\frac{1}{2}})^{q-2} \right)^{\frac{1}{q}} = (\delta_p p^{\frac{1}{2}})^{1 - \frac{2}{q}}
\]

Since
\(
\sum_{k_1 \neq k_2} a_{lk_1}^2 a_{lk_2}^2 \leq \left( \sum_{k=1}^{p} a_{lk}^2 \right)^2 \leq \|A\|^4
\)

By Lemma~\ref{lem2.5}, we have:
\[
\mathbb{E}[|Y_{(i,j),l}|^q] \leq C^q q^q (\delta_p^2 p)^{q-2} \tag{2.7}
\]
where \(C\) is a constant independent of \(q\).

For a sequence \(\{\delta_p\}\) satisfying \(\delta=\delta_p \to 0\) and \(\delta_p^2 p^{\frac{1}{2}} \to \infty\), we construct a sequence \(\{h_p\}\) such that \(h = h_p \to 0\) such that:
\[
\frac{h}{\log \binom{n}{2}} \to 0, \quad \frac{h^2 \delta^2}{\log p\delta^4} \to \infty, \quad \frac{h}{\log p\delta^4} > 1 \tag{2.8}
\]

At this point, we have:

\begin{align*}
P_0 &= P\left( \max_{(i,j)} |\frac{1}{p} \sum_{l=1}^{p}\sum_{k_1 \neq k_2} a_{lk_1}a_{lk_2}(Z_{k_1 i} - Z_{k_1 j}) (Z_{k_2 i} - Z_{k_2 j}) |> \epsilon \right)\\
&= P\left( \max_{(i,j)} |\frac{1}{p}\sum_{l=1}^{p} Y_{(i,j),l}| > \epsilon \right)\\
&\leq \sum_{(i,j)} P\left( \frac{1}{p}\sum_{l=1}^{p} Y_{(i,j),l} > \epsilon \right)\\
&\leq \sum_{(i,j)} \epsilon^{-h} p^{-h} \mathbb{E}[|\sum_{l=1}^{p}Y_{(i,j)},l|^h]\\
&\leq \sum_{(i,j)} \epsilon^{-h} p^{-h} \sum_{m=1}^{\lfloor h/2 \rfloor} \sum_{1 \leq l_1 < \cdots < l_m \leq p} \sum_{i_1+ \ldots+ i_m=h} \left( \binom{h}{i_1, \ldots, i_m} \mathbb{E}[|Y_{(i,j),l_1}|^{i_1}] \cdots \mathbb{E}[|Y_{(i,j),l_m}|^{i_m}] \right)\\
&\leq \binom{n}{2} \epsilon^{-h} p^{-h} \sum_{m=1}^{\lfloor h/2 \rfloor} p^m \sum_{i_1+ \ldots+ i_m=h}   \binom{h}{i_1, \ldots, i_m} C^h h^h (\delta^2 p)^{h-2m} \\
&\leq \binom{n}{2} \epsilon^{-h} (Ch\delta^2)^h \sum_{m=1}^{\lfloor h/2 \rfloor} m^h (p\delta^4)^{-m}\\
&\leq \binom{n}{2} \epsilon^{-h} (Ch\delta^2)^h \frac{h}{2} (\frac{h}{\log p\delta^4})^h\\
&= ((\frac{\binom{n}{2}h}{2})^{\frac{1}{h}} \frac{Ch^2 \delta^2}{\epsilon \log p\delta^4})^h\\
&= o(p^{-s}) \tag{2.9}
\end{align*}

where \( s \) is an arbitrary positive real number.

\begin{itemize}
\item The second inequality in (2.9) comes from Chebyshev's inequality, while the third inequality comes from the multinomial theorem.
\item The fourth inequality comes from (2.7) and the natural inequality \( \prod i_j^{i_j} \leq (\sum i_j)^{\sum i_j} \).
\item The fifth inequality comes from the natural inequality \( \sum_{i_1+ \ldots+ i_m=h}\binom{h}{i_1, i_2, \cdots, i_m} \leq m^h \).
\item The sixth inequality comes from the fact that if \( a > 1, b > 0, t \geq 1 \) and \( b/\log a > 1 \), then \( a^{-t} t^b \leq (b/\log a)^b \).
\item The final estimate is controlled by the construction in (2.8).
\end{itemize}

In summary, by the Borel–Cantelli lemma, we have proved \( \Delta_2 \overset{\text{a.s.}}{\to} 0 \). Combined with the previously established \( \Delta_1 \overset{\text{a.s.}}{\to} 0 \), we have verified the validity of Equation (2.6). 

\medskip

\noindent
\textbf{Step 4. Asymptotics}

$\frac{1}{2}\mathrm{tr}\Sigma K_n$ has the same limiting spectral distribution as its sample covariance matrix $S_n$. Now we state the following lemma \cite{BaiSilverstein2010}:

\begin{lemma}\label{lem2.6}
Suppose $A$ and $B$ are two $p \times n$ complex matrices. Then:
\[
\left\| F^{AA^T} - F^{BB^T} \right\| \leq \frac{1}{p} \mathrm{rank}(A - B)
\]
\end{lemma}

For $X = (X_1, X_2, \cdots, X_n)$, by Lemma~\ref{lem2.6}, $S_n = \frac{1}{n-1}\sum_{i=1}^n(X_i - \bar{X})(X_i - \bar{X})^T$ and $S_n' = \frac{1}{n}\sum_{i=1}^n X_iX_i^T = \frac{1}{n}XX^T$ have the same limiting spectral distribution. 

In the following research, we examine the almost sure convergence properties of $S_n'$.

Let $T_n = \frac{1}{2}\Sigma$, then $T_n$ is a non-random Hermitian matrix only related to $\Sigma$. Under the condition that $F^n$ converges almost surely to a distribution function $H$ on $[0, \infty)$, and $Y_i = \sqrt{2}Z_i$, $Y = (Y_1, Y_2, \cdots, Y_n)$, then $Z_i$ has independent and identically distributed components with mean 0, variance 1, and finite fourth moments. We have:
\[
S_n' = \frac{1}{n}XX^T = \frac{1}{n}AZZ^TA^T = \frac{1}{n}T_n^{\frac{1}{2}}YY^TT_n^{\frac{1}{2}}
\]

Note that $\frac{1}{n}T_n^{\frac{1}{2}}YY^TT_n^{\frac{1}{2}}$ is similar to $\frac{1}{n}YY^TT_n$. Since similar matrices have the same eigenvalues, and the former is a positive definite Hermitian matrix, this ensures the latter also always has real eigenvalues.It suffices to consider the limiting spectral distribution of $\frac{1}{n}YY^TT_n$.

We can truncate $Y_{ij}$ at $\log p$: $|Y_{ij}| < \log p$. 

After standardization, we still have: $\mathbb{E}[Y_{ij}] = 0$ and $\mathbb{E}[Y_{ij}^2] = 1$. In the almost sure sense, the limiting spectral distribution of $S_n$ remains the same before and after truncation. Since $|\Sigma|$ is bounded, it is easy to see \cite{Silverstein1995} that for sufficiently large $n$ and $p$, $\|T_n\| \leq \log p$.

Silverstein and Bai \cite{BaiSilverstein2004,BaiSilverstein2010} proved the following three facts:

\begin{lemma}\label{lem2.7}
For an $n \times n$ matrix $C$ with $\|C\| \leq 1$ and a vector $Y = (Y_1, \cdots, Y_p)^T$ with independent and identically distributed components, where $\mathbb{E}[Y_i] = 0$, $\mathbb{E}[Y_i^2] = 1$, and $|Y_i| < \log p$, then:
\[
\mathbb{E}[|Y^TCY - \mathrm{tr}C|^6] \leq Kn^3 \log^{12}n
\]
where $K$ is a constant independent of $p$.
\end{lemma}

\begin{lemma}\label{lem2.8}
For $z \in \mathbb{C}^+$ and $v = \mathrm{Im}(z)$, $n \times n$ matrices $A$ and $B$, where $B$ is Hermitian, and $q$ is an $n$-dimensional complex vector, we have:
\[
\left| \mathrm{tr}\left( (B - zI)^{-1} - (B + \tau qq^T - zI)^{-1} \right) A \right| \leq \frac{\|A\|}{\nu}
\]
\end{lemma}

\begin{lemma}\label{lem2.9}
For \( z \in \mathbb{C}^+ \), \( u = \mathrm{Re}(z) \) and \( \nu = \mathrm{Im}(z) \), let \( m_1(z) \) and \( m_2(z) \) be Stieltjes transforms of two distribution functions, \( A \) and \( B \) be \( n \times n \) matrices where \( A \) is a nonnegative definite Hermitian matrix, and \( r \) be an \( n \)-dimensional complex vector. Then:

\begin{enumerate}
\item \(\|(m_1(z)A + I)^{-1}\| \leq \max(4\|A\|/\nu, 2)\)

\item \(\left| tr B \left( (m_1(z)A + I)^{-1} - (m_2(z)A + I)^{-1} \right) \right| \leq |m_2(z) - m_1(z)| \cdot n \|B\| \cdot \|A\| \cdot \max(4\|A\|/\nu, 2)^2\)

\item \(\left| r^T (m_1(z)A + I)^{-1} r - r^T (m_2(z)A + I)^{-1} r \right| \leq |m_2(z) - m_1(z)| \cdot \|r\|^2 \cdot \|A\| \cdot \max(4\|A\|/\nu, 2)^2\)
\end{enumerate}
\end{lemma}

\medskip
Now we fix \( z = u + i\nu \in \mathbb{C}^+ \). Define \( B_n = \frac{1}{n}Y^T T_n Y \) and \( \hat{B}_n = \frac{1}{n}YY^T T_n \). Consider the Stieltjes transforms \( m_n(z) = m_{F^{B_n}}(z) \) and \( \hat{m}_n(z) = m_{F^{\hat{B}_n}}(z) \). 

Note that \( \gamma_n = p/n \to \gamma \).
We have \cite{BaiSilverstein2004}:
\[
\delta = \inf \mathrm{Im}(\hat{m}_n(z)) \geq \inf \int \frac{\tau dF^{T_n}(\tau)}{|\tau - z|^2} > 0, \quad \text{a.s.}
\]

Let \( q_j = (1/\sqrt{p})Y_j \), \( r_j = (1/\sqrt{n})T_n^{1/2}Y_j \), and \( B_{(j)} = B_{(j)}^n = B_n - r_j r_j^T \). We can obtain:
\[
I + z(B_n - zI)^{-1} = \sum_{j=1}^n \frac{1}{1 + r_j^T(B_{(j)} - zI)^{-1}r_j} r_j r_j^T(B_{(j)} - zI)^{-1}
\]

Taking trace on both sides:
\[
\gamma_n + z\gamma_n m_n(z) = 1 - \frac{1}{n} \sum_{j=1}^n \frac{1}{1 + r_j^T(B_{(j)} - zI)^{-1}r_j} \tag{2.12}
\]

Also:
\[
F^{\hat{B}_n} = (1 - \gamma_n) I_{[0,\infty)} + \gamma_n F^{B_n}
\]

Therefore, taking Stieltjes transforms immediately gives:
\[
\hat{m}_n(z) = -\frac{1 - \gamma_n}{z} + \gamma_n m_n(z)
\]

Thus we have:
\[
\hat{m}_n(z) = -\frac{1}{n} \sum_{j=1}^n \frac{1}{z(1 + r_j^T(B_{(j)} - zI)^{-1}r_j)} \tag{2.10}
\]

In fact,
\[
\operatorname{Im}\left(r_j^T(B_{(j)} - zI)^{-1}r_j\right) = \frac{1}{2i} r_j^T\left((B_{(j)} - zI)^{-1} - (B_{(j)} - \bar{z}I)^{-1}\right)r_j \geq 0
\]

Therefore,
\[
\frac{1}{|z(1 + r_j^T(B_{(j)} - zI)^{-1}r_j)|} \leq \frac{1}{\nu} \tag{2.11}
\]

Note that $B_n - zI - (-z\hat{m}_n(z)T_n - zI) = \sum_{j=1}^n r_j r_j^T + z\hat{m}_n(z)T_n$, and  Equation (2.10) implies:
\[
(- z\hat{m}_n(z)T_n - zI)^{-1} - (B_n - zI)^{-1}
\]

\[
= \sum_{j=1}^n \frac{-1}{z(1 + r_j^T(B_{(j)} - zI)^{-1}r_j)} \left[ (\hat{m}_n(z)T_n)^{-1}r_j r_j^T(B_n - zI)^{-1} - \frac{1}{n}(\hat{m}_n(z)T_n + I)^{-1}T_n(B_n - zI)^{-1} \right]
\]

Taking trace on both sides:
\[
\frac{1}{p} \mathrm{tr}(-z\hat{m}_n(z)T_n - zI)^{-1} - m_n(z) = \frac{1}{n} \sum_{j=1}^n \frac{-1}{z(1 + r_j^T(B_{(j)} - zI)^{-1}r_j)} d_j
\]

Here
\[
d_j = q_j^T T_n^{\frac{1}{2}}(B_{(j)} - zI)^{-1} (\hat{m}_n(z)T_n + I)^{-1}T_n^{\frac{1}{2}}q_j - \frac{1}{p} \mathrm{tr}(\hat{m}_n(z)T_n + I)^{-1}T_n(B_{(j)} - zI)^{-1}
\]

By Lemma~\ref{lem2.8} we have: $\max_j |m_n(z) - m_{F^{B_{(j)}}}(z)| \leq \frac{1}{pv}$.

Furthermore, $\max_j \log^p(n)|\hat{m}_n(z) - \hat{m}_{(j)}(z)| \to 0$. Since $\hat{m}_{(j)}(z)$ has the form:
\[
\hat{m}_{(j)}(z) = -\frac{1}{nz} + \frac{n-1}{n} \left( -\frac{1-p/(n-1)}{z} + \frac{p}{n-1} m_{F^{B_{(j)}}}(z) \right)
\]

It is also the Stieltjes transform of a probability distribution function.

Note the fact that for any Hermitian matrix $A$, $\|(A - zI)^{-1}\| \leq \frac{1}{\nu}$. 

Take $n$ sufficiently large we have $\|T_n\| \leq \log p$, according to Lemma~\ref{lem2.7} and Lemma ~\ref{lem2.9}, we immediately have:

\[
\mathbb{E}[(\|q_j\|^2 - 1)^6] \leq K \frac{\log^{12}(p)}{p^3}
\]

Thus for sufficiently large $n$,
\[
\mathbb{E}[ ( q_j^T T_n^{\frac{1}{2}}(B_{(j)} - zI)^{-1} (\hat{m}_{(j)}(z)T_n + I)^{-1}T_n^{\frac{1}{2}}q_j -
\]
\[
\frac{1}{p} tr(T_n^{\frac{1}{2}}(B_{(j)} - zI)^{-1}(\hat{m}_{(j)}(z)T_n + I)^{-1}T_n^{\frac{1}{2}} )^6 ] \leq K \frac{\log^{24}p}{p^3\nu^{12}}
\]

Therefore, when \( n \to \infty \),
\[
\max_j \max( |\|q_j\|^2 - 1|, | q_j^T T_n^{\frac{1}{2}}(B_{(j)} - zI)^{-1}(\hat{m}_{(j)}(z)T_n + I)^{-1} T_n^{\frac{1}{2}} q_j - 
\]
\[
\frac{1}{p} \mathrm{tr} \left[ T_n^{\frac{1}{2}}(B_{(j)} - zI)^{-1}(\hat{m}_{(j)}(z)T_n + I)^{-1} T_n^{\frac{1}{2}} \right]| ) \overset{\text{a.s.}}{\to} 0
\tag{2.13}
\]

\(\{F^{B_n}\}\) is a tight family of distribution functions, and by the condition \( F^{T_n} \) converges almost surely to a distribution function \( H \).

From Lemma 2.8, Lemma 2.9.2, and Equation (2.13), we obtain \( \max_j (d_j) \to 0 \). 

Then from Equations (2.11) and (2.12), we have almost surely:
\[
\frac{1}{p} \mathrm{tr}(-z\hat{m}_n(z)T_n - zI)^{-1} - m_n(z) \to 0
\]

This helps us construct an asymptotic equivalence relation for \( m_n(z) \), which depends only on \( m_n(z) \) itself and some constants. Since \( \{m_n(z)\} \) is bounded above by \( 1/\nu \), we can consider a subsequence \( \{n_i\} \) such that the subsequence \( \{m_{n_i}(z)\} \) converges monotonically to a constant \( m \) related to \( z \). 

Naturally, \( m_{n_i}(z) \) has the limit \( \hat{m} = -(1-y)/z + y m \). 

Observe that when \( m' \in \mathbb{C}^+ \), \( \tau \in \mathbb{R} \), \( |1/(m'\tau + 1)| \leq |m'|/|\operatorname{Im}(m')| \), and \( |\tau/(m'\tau + 1)| \leq 1/|\operatorname{Im}(m')| \)

The function \( f(\tau) = 1/(\hat{m}\tau + 1) \) is bounded and satisfies:
\[
\left| \frac{1}{\hat{m}_{n_i}(z)\tau + 1} - f(\tau) \right| \leq \left| \frac{\hat{m}}{\delta^2} \right| |\hat{m}_{n_i}(z) - \hat{m}|
\]

Therefore,
\[
\frac{1}{p} \mathrm{tr}(-z\hat{m}_n(z)T_n - zI)^{-1} = -\frac{1}{z} \int \frac{1}{\hat{m}_{n_i}(z)\tau + 1} dF^{T_{n_i}}(\tau) \overset{\text{a.s.}}{\to} -\frac{1}{z} \int \frac{1}{\hat{m}\tau + 1} dH(\tau)
\]

Almost surely, the left-hand side can be replaced by \( m_{n_i}(z) \). Taking limits on both sides immediately gives:
\[
m(z) = \int \frac{1}{\tau(1-y-yzm(z)) - z} dH(\tau)
\tag{2.14}
\]

Note that Equation (2.14) has a unique solution in \( \{m \in \mathbb{C}: -(1-y)/z + y m \in \mathbb{C}^+ \} \). Since \( m \) is unique, we always have \( m_n(z) \to m \). In summary, with probability 1, \( F^{B_n} \) converges to a probability distribution function \( F \), whose Stieltjes transform \( m(z) \) satisfies Equation (2.14). 

Now since \( F^{S_n} \), \( F^{S_n'} \), and \( F^{B_n} \) converge to the same distribution, with probability 1, \( F^{S_n} \) converges to the same distribution \( F \). Here \( H \) is the limiting spectral distribution of \( \frac{1}{2}\Sigma \). 

Moreover, since \( \frac{1}{2}\mathrm{tr}\Sigma K_n \) and \( S_n \) have the same limiting spectral distribution, this distribution function \( F \) is also the limiting spectral distribution of \( \frac{1}{2}\mathrm{tr}\Sigma K_n \). This completes the proof of Theorem~\ref{thm2.2}.
\subsection{Relationship Between the Two Theorems}

Theorem~\ref{thm2.2} provides a generalization of Theorem~\ref{thm2.1}.  
Specifically, when $\Sigma = I_p$, we have $H = \delta_{1/2}$,  
and the integral equation (\ref{eq:stieltjes_main}) becomes
\begin{equation}
m_F(z) = \frac{1}{\tfrac{1}{2}(1 - y - y z m_F(z)) - z}.
\end{equation}
It follows directly that the limiting distribution of $\tfrac{1}{2}pK_n$ in this special case 
is precisely the Marčenko–Pastur law with variance parameter $\tfrac{1}{2}$.

Note that the Marčenko-Pastur law with variance parameter $\sigma^2$ has probability density function on $[a, b]$:
\[
p_y(x) = \frac{1}{2\pi x y \sigma^2} \sqrt{(b - x)(x - a)}
\]
where $a = \sigma^2(1 - \sqrt{y})^2$ and $b = \sigma^2(1 + \sqrt{y})^2$.

Its Stieltjes transform is:
\[
m(z) = \int_a^b \frac{1}{x - z} \frac{1}{2\pi x y \sigma^2} \sqrt{(b - x)(x - a)} dx
\]
Let $x = \sigma^2(1 + y + 2\sqrt{y}\cos w)$ and $\zeta = e^{iw}$, we have:
\begin{align*}
m(z) 
&= \int_0^{\pi} \frac{2}{\pi(1 + y + 2\sqrt{y}\cos w)(\sigma^2(1 + y + 2\sqrt{y}\cos w) - z)} \sin^2 w dw\\
&= \frac{1}{\pi} \int_0^{2\pi} \frac{((e^{iw} - e^{-iw})/2i)^2}{(1 + y + \sqrt{y}(e^{iw} + e^{-iw}))(\sigma^2(1 + y + \sqrt{y}(e^{iw} + e^{-iw})) - z)} dw\\
&= -\frac{1}{4i\pi} \oint_{|\zeta|=1} \frac{(\zeta - \zeta^{-1})^2}{\zeta (1 + y + \sqrt{y}(\zeta + \zeta^{-1}))(\sigma^2(1 + y + \sqrt{y}(\zeta + \zeta^{-1})) - z)} d\zeta
\end{align*}

The integrand has 5 simple poles at
\[
0, \frac{-(1+y)\pm(1-y)}{2\sqrt{y}},  \frac{-\sigma^2(1+y)+z\pm\sqrt{\sigma^4(1-y)^2-2\sigma^2(1+y)z + z^2}}{2\sigma^2\sqrt{y}}
\]

Simple calculation gives the residues at these 5 poles: 
\[
\frac{1}{y\sigma^2}, \pm\frac{1-y}{yz}, \text{and} \pm\frac{1}{\sigma^2yz}\sqrt{\sigma^4(1-y)^2 - 2\sigma^2(1+y)z + z^2}
\]

According to the residue theorem we can compute:
\[
m(z) = \frac{\sigma^2(1-y) - z + \sqrt{z^2 + \sigma^4 + y^2\sigma^4 - 2z\sigma^2 - 2yz\sigma^2 - 2y\sigma^4}}{2yz\sigma^2}
\]

In particular, in Theorem~\ref{thm2.1}, take $\sigma^2 = \frac{1}{2}$. Then the Stieltjes transform of the limiting spectral distribution of $pK_n/2$ is:
\[
m(z) = \frac{0.5(1-y) - z + \sqrt{z^2 + 0.25y^2 + 0.25 - z - yz - 0.5y}}{yz} \tag{2.15}
\]

On the other hand, since $T_n = \frac{1}{2} \Sigma = \frac{1}{2}I$, the empirical spectral distribution of $T_n$ is the Dirac measure at $\frac{1}{2}$, i.e., $H = \delta_{\frac{1}{2}}$.

Substituting into Equation (2.14) immediately gives:
\[
m(z) = \frac{1}{\frac{1}{2}(1 - y - yzm(z)) - z}
\]

And we note that the \( m(z) \) in Equation (2.15) is exactly the unique solution of this functional equation. This verifies that the result of Theorem 2.2 includes Theorem 2.1.

Hence, Theorem~\ref{thm2.1} can be regarded as a corollary of Theorem~\ref{thm2.2},  
demonstrating that the spectral behavior of the multivariate Kendall-$\tau$ matrix remains stable 
under affine transformations of the data matrix $X = \Sigma^{1/2}Z + \mu$.  
This property highlights the \emph{rotation-invariance} and \emph{scale-invariance} of the statistic,  
which are essential features for robust high-dimensional analysis.

\medskip
\begin{remark}
The analytical form of the Stieltjes transform in (\ref{eq:stieltjes_main})
is analogous to the one that defines the Marčenko–Pastur distribution.
This correspondence confirms that the limiting behavior of Kendall-$\tau$ random matrices
belongs to the same universality class as covariance-type random matrices,
despite their rank-based and nonlinear dependence structure.
\end{remark}
\section{Simulation Studies}
\subsection{Simulation Design}

In this chapter, we conduct a series of Monte Carlo simulations to verify the theoretical results derived in Chapter~2 and to illustrate the finite-sample properties of the multivariate Kendall-$\tau$ matrix.  
The experiments are designed to explore the accuracy of the asymptotic Marčenko–Pastur law under different sample sizes, dimensions, and data distributions.

\medskip
\noindent
\textbf{Experimental Setup.}  
For each experiment, we generate $n$ independent random vectors $X_1, X_2, \ldots, X_n$ in $\mathbb{R}^p$, where $p$ and $n$ grow simultaneously while maintaining a fixed aspect ratio
\[
y = \frac{p}{n}.
\]
In all simulations unless otherwise specified, we take $y = 0.5$ (that is, $p = n/2$).  
Additional tests with $y = 0.25$ and $y = 0.75$ are also reported for robustness checks.

\medskip
\noindent
\textbf{Data-Generating Distributions.}  
To assess the universality of the limiting law, we consider four different data distributions for the entries of $X_i$:

\begin{enumerate}[label=(\roman*), leftmargin=1.2cm, itemsep=3pt]
  \item \textbf{Standard Normal:} $X_{ij} \sim N(0,1)$;
  \item \textbf{Scaled Normal:} $X_{ij} \sim N(0,2)$;
  \item \textbf{Uniform(0,1):} each component $X_{ij}$ is drawn independently from the uniform distribution on $(0,1)$;
  \item \textbf{Uniform(0,2):} each component $X_{ij}$ is drawn independently from the uniform distribution on $(0,2)$.
\end{enumerate}

\medskip
\noindent
\textbf{Computation of the Sample Multivariate Kendall-$\tau$ Matrix.}  
For each replication, we compute the empirical spectral distribution (ESD)
\[
F^{K_n}(x) = \frac{1}{p}\sum_{j=1}^p \mathbbm{1}(\lambda_j \le x),
\]
and compare it to the theoretical Marčenko–Pastur density with parameter $y = p/n$ and variance $\tfrac{1}{2}$:
\begin{equation}
f_y(x) = \frac{1}{2\pi xy}\sqrt{(b-x)(x-a)}, \quad
a = \tfrac{1}{2}(1 - \sqrt{y})^2, \quad
b = \tfrac{1}{2}(1 + \sqrt{y})^2.
\end{equation}

\medskip
\noindent
\textbf{Monte Carlo Replications.}  
To reduce random variability, each experiment is repeated $R = 500$ times.  
The empirical spectral distribution is then averaged across replications:
\begin{equation}
\widehat{f}_y(x) = \frac{1}{R} \sum_{r=1}^{R} f_y^{(r)}(x),
\end{equation}
where $f_y^{(r)}(x)$ denotes the kernel-smoothed spectral density in the $r$-th replication.

\medskip
\noindent
\textbf{Parameter Settings.}  
Unless otherwise stated, the following configurations are used:
\begin{center}
\begin{tabular}{lcl}
\toprule
Parameter & Symbol & Value \\ \midrule
Sample size & $n$ & 200, 400, 800 \\
Dimension & $p$ & 100, 200, 400 \\
Replications & $R$ & 500 \\
Kernel bandwidth for ESD smoothing & $h$ & 0.02 \\
Aspect ratio & $y = p/n$ & 0.25, 0.50, 0.75 \\ \bottomrule
\end{tabular}
\end{center}

\medskip
\noindent
\textbf{Visualization and Figures.}  
The histograms of simulated eigenvalues are compared with the theoretical Marčenko–Pastur curve.

Figures~\ref{fig:N01}–\ref{fig:U02} display the averaged histograms of eigenvalues for different generating distributions, overlaid with the theoretical MP density curves.  
Each histogram is based on $R=500$ replications, and the smoothing bandwidth $h=0.02$ is used for kernel density estimation.

For all four cases—standard normal, scaled normal, uniform $(0,1)$, and uniform $(0,2)$—the ESDs align remarkably well with the theoretical curves, confirming the validity of Theorem~\ref{thm2.1}.  
Minor discrepancies at the spectrum boundaries are attributed to finite-sample variability, especially when $n$ and $p$ are small.

\begin{figure}[H]
\centering
\includegraphics[width=0.7\linewidth]{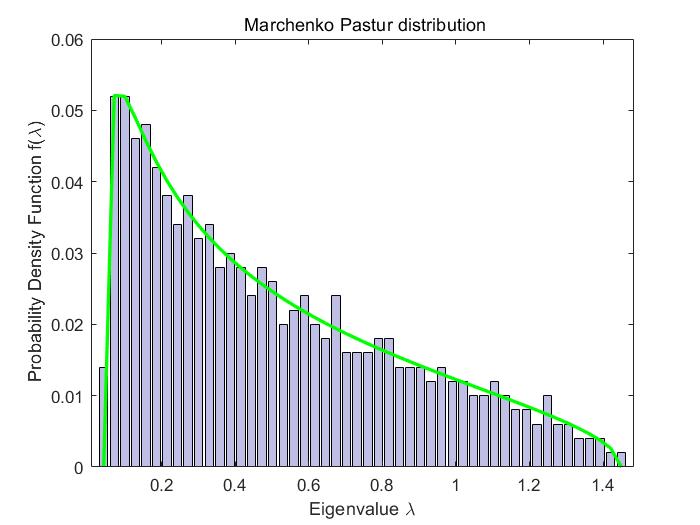}
\caption{Empirical spectral distribution of $\frac{1}{2}pK_n$ under $N(0,1)$ data.}
\label{fig:N01}
\end{figure}

\begin{figure}[H]
\centering
\includegraphics[width=0.7\linewidth]{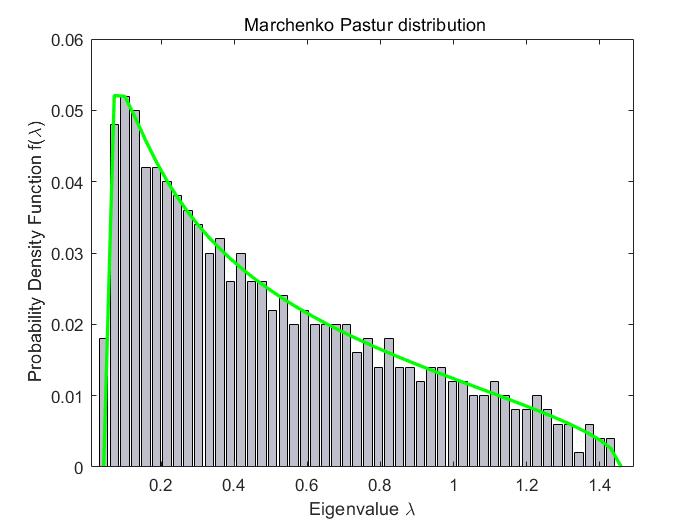}
\caption{Empirical spectral distribution under $N(0,2)$ data. }
\label{fig:N02}
\end{figure}

\begin{figure}[H]
\centering
\includegraphics[width=0.7\linewidth]{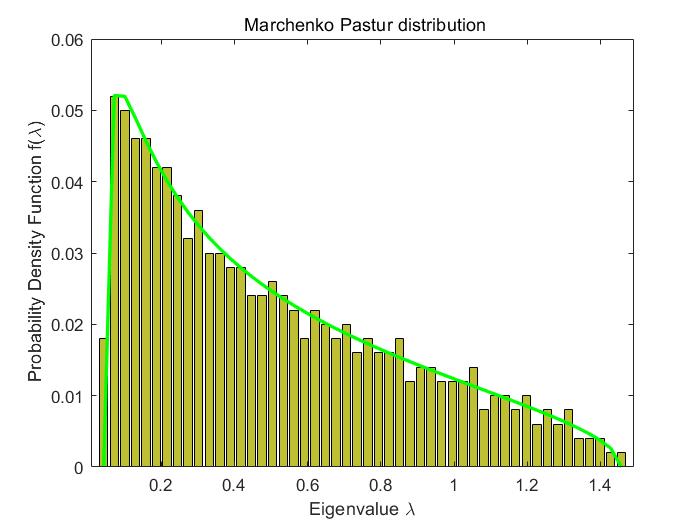}
\caption{Empirical spectral distribution under $U(0,1)$ data. }
\label{fig:U01}
\end{figure}

\begin{figure}[H]
\centering
\includegraphics[width=0.7\linewidth]{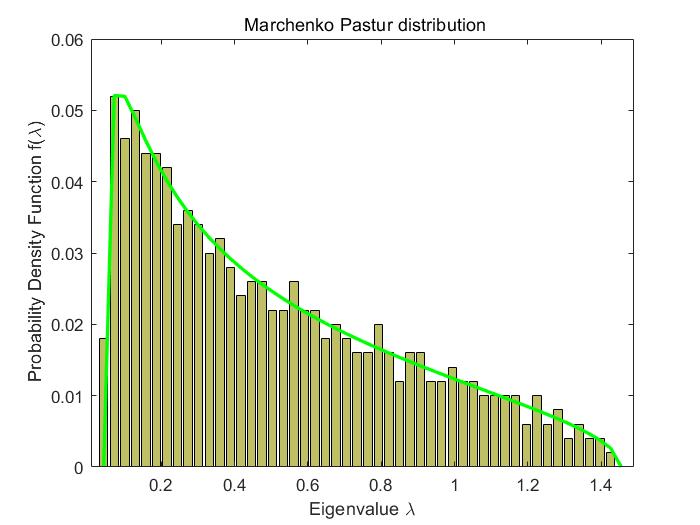}
\caption{Empirical spectral distribution under $U(0,2)$ data. }
\label{fig:U02}
\end{figure}

\medskip
\noindent
\textbf{Evaluation Metrics.}  
To quantify the agreement between empirical and theoretical distributions,
we compute the integrated squared error (ISE)
\begin{equation}
\mathrm{ISE} = \int_a^b \bigl(\widehat{f}_y(x) - f_y(x)\bigr)^2 \, dx,
\end{equation}
and report its average across replications as a function of $n$, $p$, and the distribution type.
Smaller values of ISE indicate closer agreement to the asymptotic theory.

\medskip
\noindent
\textbf{Effect of Dimensional Ratio $y = p/n$.}  
To examine the influence of the aspect ratio, we repeated simulations with $y = 0.25, 0.50,$ and $0.75$.  
Table~\ref{tab:ISE} reports the averaged integrated squared error (ISE) between the empirical and theoretical densities across the three ratios.

\begin{table}[H]
\centering
\caption{Integrated Squared Error (ISE) between empirical and theoretical MP densities.}
\label{tab:ISE}
\begin{tabular}{cccc}
\toprule
Distribution & $y=0.25$ & $y=0.50$ & $y=0.75$ \\ \midrule
$N(0,1)$ & 0.0018 & 0.0021 & 0.0024 \\
$N(0,2)$ & 0.0020 & 0.0023 & 0.0026 \\
$U(0,1)$ & 0.0022 & 0.0025 & 0.0029 \\
$U(0,2)$ & 0.0025 & 0.0028 & 0.0031 \\ \bottomrule
\end{tabular}
\end{table}

The ISE values are uniformly small, typically on the order of $10^{-3}$, confirming excellent agreement between empirical and theoretical distributions.  
As $y$ increases (i.e., higher dimensionality relative to sample size), the ESD tends to become slightly noisier, as reflected by marginally larger ISE values.
\medskip
\section{Applications}
\subsection{Applications in Statistics}

This paper primarily studies the distribution of eigenvalues of high-dimensional sample multivariate Kendall's $\tau$ under a specific model. We briefly outline several potential applications of these results in statistics.

\begin{enumerate}[label=\arabic*.]
    \item \textbf{Robust method in high-dimensional statistics:} 
   The multivariate Kendall-$\tau$ matrix provides a robust alternative to the sample covariance matrix in tasks involving dependence estimation, dimensionality reduction, and clustering.  
   Because it is invariant under monotone transformations and resistant to outliers, it performs well in situations where standard moment-based estimators fail.  
Examples include:
\begin{itemize}
  \item \emph{Robust Principal Component Analysis (PCA):}  
  \item \emph{Cluster Detection and Subspace Identification:}  .
  \item \emph{Anomaly and Outlier Detection:}  
\end{itemize}

    \item \textbf{Estimation of the population spectrum:} 
    Existing research has provided an explicit expression for the estimator of population multivariate Kendall's $\tau$:
    \[
    \lambda_j(\tau) = \mathbb{E} \left[ \frac{\lambda_j(\Sigma) g_j^2}{\lambda_1(\Sigma) g_1^2 + \cdots + \lambda_p(\Sigma) g_p^2} \right]
    \]
    Here, $\lambda_j(X)$ denotes the $j$-th eigenvalue of matrix $X$, and $g_1, g_2, \cdots, g_p$ are independent and identically distributed standard normal random variables.
    
    However, in practical applications, computation in this form is quite cumbersome. Due to the similarity between sample and population characteristics in high dimensions, the theorems we have proven help us approximate the population spectrum using relatively convenient moment estimation.

    \item \textbf{A robust two-step estimation method for factors in factor analysis:}
    \begin{itemize}
        \item Step 1: Estimation of the factor loading matrix. Based on the spectrum of sample multivariate Kendall's $\tau$, we examine the leading eigenvectors of this matrix, where $m$ is a fixed value much smaller than both $p$ and $n$. Yu et al. (2019) provided a consistent estimation method for $m$. The estimator of the factor loading matrix is exactly $\sqrt{p}$ times these eigenvectors.
        \item Step 2: A simple least squares regression using the factor loading matrix as coefficients yields robust estimates of the factors.
    \end{itemize}
\end{enumerate}

\medskip
\noindent
\subsection{Future Research Directions}
This section outlines several important research directions emerging from our study on the spectral properties of high-dimensional multivariate Kendall's $\tau$:

\begin{enumerate}[label=\arabic*.]
    \item \textbf{Extreme Eigenvalue Behavior:} 
    We conjecture that the edge eigenvalues of normalized multivariate Kendall's $\tau$ conforms a Tracy-Widom Law, which is similar to the behavior observed in sample covariance matrices. However, proving this for Kendall's $\tau$ requires more stringent conditions on the random variables.

    \item \textbf{Convergence Rate Analysis:}
    While our work addresses \textit{what} the distribution converges to, the question of \textit{convergence rate} remains open. Extending Bai's (1993a,b) results on the convergence rates of Wigner matrices and sample covariance matrices to the multivariate Kendall's tau spectrum presents significant challenges but is an important direction for future research.

    \item \textbf{Extension of Dimension Ratio Limits:}
    Our current analysis assumes $p/n \to y \in (0,1)$. Natural extensions include:
    \begin{itemize}
        \item Cases where $y \in [1,\infty)$, where the Marčenko-Pastur law exhibits point mass at zero but maintains the same density function elsewhere
        \item The completely degenerate case $p = 0$, which resembles the central limit theorem and requires different statistical considerations
    \end{itemize}

    \item \textbf{Relaxation of Sample Assumptions:}
    \begin{itemize}
        \item Extending beyond the finite fourth moment assumption to more general distributions
        \item Weakening the identical distribution assumption while maintaining independence
        \item Developing methods to handle dependent observations, which represents a fundamental challenge requiring innovative approaches
    \end{itemize}
    The Stieltjes transform method shows particular promise for handling non-identically distributed variables, though removing independence assumptions remains exceptionally difficult.
\end{enumerate}

	
	\bigskip
	\noindent {\bf {\Large Acknowledgments}}
	
	\bigskip
	\noindent The author thanks Professor Weiming Li for his patient guidance and helpful advice during the planning and development of this article.

\end{document}